\documentclass[a4paper]{amsart}
\usepackage{mathrsfs}
\usepackage{bbm}
\usepackage{latexsym, amssymb,amsmath,amsthm}
\usepackage[all, knot]{xy}
\xyoption{arc}

\def \To{\longrightarrow}

\def \Q{\operatorname{Q}}
\def \bi{\bowtie}
\def \q{\mathbbm{q}}

\def \Ker{\operatorname{Ker}}

\def \k{\kappa}

\def \D{\Delta}

\def \e{\varepsilon}
\def \M{\mathrm{M}}

\def \unit{\mathbf{1}}

\numberwithin{equation}{section}

\newtheorem{theorem}{Theorem}[section]
\newtheorem{lemma}[theorem]{Lemma}
\newtheorem{proposition}[theorem]{Proposition}
\newtheorem{corollary}[theorem]{Corollary}

\newtheorem{remark}[theorem]{Remark}

\begin{document}
\title{The Quasi-Hopf analogue of $\mathbf{u}_{q}(\mathfrak{sl}_{2})$}
\author{Gongxiang Liu}
\address{Department of Mathematics, Nanjing University, Nanjing 210093, China} \email{gxliu@nju.edu.cn}
\maketitle

\begin{abstract} In \cite{Gel}, some quasi-Hopf algebras of dimension $n^{3}$, which can be understood as the quasi-Hopf analogues of
 Taft algebras,  are constructed. Moreover, the quasi-Hopf analogues of generalized Taft algebras are  considered in \cite{HLY}, where the language of
 the dual of a quasi-Hopf algebra is used.  The Drinfeld doubles of such quasi-Hopf algebras are computed in this paper. The authors in \cite{EG} shew that the
 Drinfeld double of a quasi-Hopf algebra of dimension $n^{3}$ constructed in \cite{Gel} is always twist equivalent to Lusztig's small quantum group $\mathbf{u}_{q}(\mathfrak{sl}_{2})$ if $n$ is odd.
 Based on computations and analysis, we show that this is \emph{not} the case if $n$ is even. That is, the quasi-Hopf analogue $\Q\mathbf{u}_{q}(\mathfrak{sl}_{2})$
  of $\mathbf{u}_{q}(\mathfrak{sl}_{2})$ is gotten.
\end{abstract}

\section{Introduction}

Historically, the Lusztig's definition of a quantum group \cite{Lus} opens a convenient door for a pure mathematician to go into
the field of quantum groups. Do we have such definition for a quasi-quantum group? For a simple finite-dimensional
 Lie algebra $\mathfrak{g}$ over $\mathbb{C}$, Drinfeld [Proposition 3.16, 3] told us
that a quatitriangular quasi-Hopf quantized enveloping algebra $U\mathfrak{g}[[h]]$ is indeed twist equivalent to the usual quantum group $U_{h}\mathfrak{g}$.
So there is no really quasi-quantum group attached to a simple finite-dimensional Lie algebra. But, how about the restricted case? That is, do we have quasi-Hopf analogue of Lusztig's
definition of a small quantum group?

The aim of the paper and following works is to find the quasi-Hopf analogues of Lusztig's small quantum groups and consequently give some new examples
of finite dimensional nonsemisimple quasitriangular quasi-Hopf algebras. As a try, we want to give the quasi-Hopf analogue of
$\mathbf{u}_{q}(\mathfrak{sl}_{2})$
in this paper. Inspired by the Hopf case, one can believe that it should be the Drinfeld double of the quasi-Hopf analogue of a Taft algebra.
Meanwhile, the general theory
of the Drinfeld double for a quasi-Hopf algebra was already developed by Majid, Hausser-Nill and Schauenburg \cite{Maj,FN,Sch} and
the quasi-Hopf analogues, denoted by $A(n,q)$, of Taft algebras were discovered by
Gelaki \cite{Gel}.  So all things were prepared, and the only task is to compute them out.

But, before computation, Etingof and Gelaki shew that almost nothing new will be created \cite{EG}! Precisely, as one conclusion of
the main result in \cite{EG},  they proved that the double $D(A(n,q))$ is always twist equivalent
to $\mathbf{u}_{q}(\mathfrak{sl}_{2})$ if $n$ is odd.
 There is a restriction on Etingof-Gelaki's result,  that is, $n$ must be odd. Is this condition necessary? Our answer is YES.
 As one of main results of this paper, we show that $D(A(n,q))$  is \emph{not} twist equivalent to a Hopf algebra
 if $n$ is even and consequently the quasi-Hopf analogue of $\mathbf{u}_{q}(\mathfrak{sl}_{2})$ is gotten.  We will prove the result in a general setting.

 In \cite{HLY}, all pointed Majid algebras $M(n,s,q)$ of finite representation type are classified.
  Such pointed Majid algebras are indeed the dual of the class of basic quasi-Hopf algebras $A(n,s,q)$
 which can be considered as the
 quasi-Hopf analogues of generalized Taft algebras. Note that the quasi-Hopf algebras $A(n,s,q)$ also appeared in \cite{Ang}.
 Maybe, the only contribution of this paper is
 to compute $D(A(n,s,q))$ out explicitly. The main result of paper can be described as follows.

 \begin{theorem}\emph{ (1)} As a quasi-Hopf algebra, $D(A(n,s,q))\cong \Q_{s}\mathbf{u}_{q}(\mathfrak{sl}_{2})$ .

 \emph{(2)} Assume that $n=2^{m}l$ and $s=2^{m'}l'$ with $(l,2)=(l',2)=1$. If $m'<m$, then $D(A(n,s,q))$ is not twist equivalent to a Hopf algebra.
 \end{theorem}

  The quasi-Hopf algebra $\Q_{s}\mathbf{u}_{q}(\mathfrak{sl}_{2})$ will be described in Section 2 by using generators and relations.
  All other preliminaries are also collected in this section.
  The first part of Theorem 1.1 will be proved in Section 3 and the method is  computation.
  The proof of the second part will be given in Section 4. The main idea of this section is to find
  some suitable representations of $\Q_{s}\mathbf{u}_{q}(\mathfrak{sl}_{2})$  such that they form
  a subtensor category of Rep-$\Q_{s}\mathbf{u}_{q}(\mathfrak{sl}_{2})$. By using group cohomologies,
  we will find that the restriction of the reassociator to this subtensor category is not trivial.

   Throughout, we work over an algebraically closed field $\kappa$ of characteristic $0$ and
  $[\frac{\bullet}{\bullet}]$ denote the Guassian fraction function. That is, for any natural numbers
  $a,b$, $[\frac{a}{b}]$ denotes the biggest integer which is not
  bigger than $\frac{a}{b}$.  About general background
   knowledge, the reader is referred to \cite{Dri} for
   quasi-Hopf algebras, to \cite{CE,kassel} for general theory about tensor categories,  and to \cite{HLY} for pointed Majid algebras.

\section{Preliminaries} In this section we recall the constructions of quasi-Hopf analogues of (generalized) Taft algebas,
their dualities and the Drinfeld double
of a quasi-Hopf algebra for the convenience of the reader. At last, we will introduce a new quasi-Hopf algebra $\Q_{s}\mathbf{u}_{q}(\mathfrak{sl}_{2})$.

\subsection{Path coalgebras and pointed Majid algebra $M(n,s,q)$.}
The main aim of this subsection is to recall the definition of the pointed Majid algebra $M(n,s,q)$ constructed in
\cite{HLY}. To attack it, the concept path coalgeba is needed.

A quiver is a quadruple $Q=(Q_0,Q_1,s,t),$ where $Q_0$ is the set of
vertices, $Q_1$ is the set of arrows, and $s,t:\ Q_1 \longrightarrow
Q_0$ are two maps assigning respectively the source and the target
for each arrow. A path of length $l \ge 1$ in the quiver $Q$ is a
finitely ordered sequence of $l$ arrows $a_l \cdots a_1$ such that
$s(a_{i+1})=t(a_i)$ for $1 \le i \le l-1.$ By convention a vertex is
said to be a trivial path of length $0.$
For a quiver $Q,$ the associated path coalgebra $\k Q$ is the
$\k$-space spanned by the set of paths with counit and
comultiplication maps defined by $\e(g)=1, \ \D(g)=g \otimes g$ for
each $g \in Q_0,$ and for each nontrivial path $p=a_n \cdots a_1, \
\e(p)=0,$
\begin{equation}
\D(a_n \cdots a_1)=p \otimes s(a_1) + \sum_{i=1}^{n-1}a_n \cdots
a_{i+1} \otimes a_i \cdots a_1 \nonumber + t(a_n) \otimes p \ .
\end{equation}
The length of paths gives a natural gradation to the path coalgebra.
Let $Q_n$ denote the set of paths of length $n$ in $Q,$ then $\k
Q=\oplus_{n \ge 0} \k Q_n$ and $\D(\c Q_n) \subseteq
\oplus_{n=i+j}\c Q_i \otimes \c Q_j.$ Clearly $\k Q$ is pointed with
the set of group-likes $G(\c Q)=Q_0,$ and has the following
coradical filtration $$ \k Q_0 \subseteq \k Q_0 \oplus \k Q_1
\subseteq \k Q_0 \oplus \k Q_1 \oplus \k Q_2 \subseteq \cdots.$$
Hence $\k Q$ is coradically graded.

A dual quasi-bialgebra, or Majid bialgebra for short, is a coalgebra
$(M,\D,\e)$ equipped with a compatible quasi-algebra structure.
Namely, there exist two coalgebra homomorphisms $$\M: H \otimes H
\to H, \ a \otimes b \mapsto ab, \quad \mu: \k \to H,\ \lambda
\mapsto \lambda 1_H$$ and a convolution-invertible map $\Phi:
H^{\otimes 3} \to k$ called reassociator, such that for all $a,b,c,d
\in H$ the following equalities hold:
\begin{gather}
a_{(1)}(b_{(1)}c_{(1)})\Phi(a_{(2)},b_{(2)},c_{(2)})=\Phi(a_{(1)},b_{(1)},c_{(1)})(a_{(2)} b_{(2)})c_{(2)},\\
1_Ha=a=a1_H, \\
\Phi(a_{(1)},b_{(1)},c_{(1)}d_{(1)})\Phi(a_{(2)}b_{(2)},c_{(2)},d_{(2)}) \\
 =\Phi(b_{(1)},c_{(1)},d_{(1)})\Phi(a_{(1)},b_{(2)}c_{(2)},d_{(2)})\Phi(a_{(3)},b_{(1)},c_{(3)}),\nonumber \\
\Phi(a,1_H,b)=\e(a)\e(b).
\end{gather}
Here and below we use the Sweedler sigma notation $\D(a)=a_{(1)} \otimes
a_{(1)}=a'\otimes a''$ for the comultiplication and $a_{(1)}\otimes a_{(2)}\otimes \cdots \otimes a_{(n+1)}$
 for the result of the $n$-iterated application of $\D$ on $a$. $H$ is called a Majid algebra if, moreover,
there exist a coalgebra antimorphism $S: H \to H$ and two
functionals $\alpha,\beta: H \to \k$ such that for all $a \in H,$
\begin{gather}
S(a_{(1)})\alpha(a_{(2)})a_{(3)}=\alpha(a)1_H, \quad
a_{(1)}\beta(a_{(2)})S(a_{(3)})=\beta(a)1_H, \\
\Phi(a_{(1)},S(a_{(3)}), a_{(5)})\beta(a_{(2)})\alpha(a_{(4)})= \\
\Phi^{-1}(S(a_{(1)}),a_{(3)},S(a_{(5)})) \alpha(a_{(2)})\beta(a_{(4)})=\e(a).
\nonumber
\end{gather}

A Majid algebra $H$ is said to be pointed, if the underlying
coalgebra is pointed. Now we consider a very simple quiver and hope to build a pointed Majid algebra structure on its path coalgebra.
The quiver being considered is the following one.
$$ 
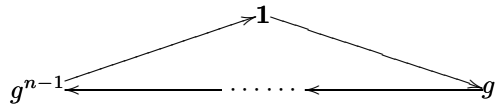
\xy {\ar (0,0)*{\unit}; (30,-10)*{g}}; {\ar (-30,-10)*{g^{n-1}};
(0,0)*{\unit}}; {\ar (30,-10)*{g}; (3,-10)*{ \cdots  \ }}; {\ar
(-3,-10)*{\ \cdots  }; (-30,-10)*{g^{n-1}}}
\endxy $$

As in \cite{HLY}, this quiver is denoted by $Q(\mathbb{Z}_{n},g)$. Now let $0\leq s\leq n-1$ be a natural number which is a factor of $n$, i.e., $s|n$, $q$ an $n^{2}$-th primitive root of unity and $\q:=q^{n}$.
Let $p_{i}^{l}$ denote the path in $Q(\mathbb{Z}_{n},g)$ starting from $g^{i}$ with length $l$. So $p_{i}^{0}=g^{i}$.
Let $\Phi_{s}$ be the 3-cocycle over $\mathbb{Z}_{n}$ defined by
\begin{equation}\Phi_{s}(g^{i},g^{j},g^{k})=\q^{si[\frac{j+k}{n}]},\;\;\;0\leq i,j,k\leq  n-1.
\end{equation}
To define $M(n,s,q)$, the definition of the Guassian binomial coefficient
is needed. For any $\hbar \in \k$,
define $l_\hbar=1+\hbar+\cdots +\hbar^{l-1}$ and $l!_\hbar=1_\hbar
\cdots l_\hbar$. The Gaussian binomial coefficient is defined by
$\binom{l+m}{l}_\hbar:=\frac{(l+m)!_\hbar}{l!_\hbar m!_\hbar}$.

Now we can define the pointed Majid algebra $M(n,s,q)$. As a coalgebra, $M(n,s,q)=\oplus_{i< n[\frac{n}{s}]} \k Q(\mathbb{Z}_{n},g)_{i}$.
The reassociator, the multiplication, the functions $\alpha,\beta$ and the antipode are given through
\begin{gather}\Phi(p_i^{l},p_j^{m}, p_k^{t})=\delta_{lmt,0}\Phi_{s}(g^{i},g^{j},g^{k}),\\
p_i^{l}\cdot p_j^{m}=q^{-sjl}\q^{s(i+l)[\frac{m+j}{n}]}
\binom{l+m}{l}_{q^{-s}}p_{i+j}^{l+m},\\
\alpha(p_i^{l})=\delta_{l,0}\frac{1}{\Phi_{s}(g^{i},g^{n-i},g^{i})},\;\;\;\;\beta(p_i^{l})=\delta_{l,0}1,\\
S(g^{i})=g^{n-i},\;\;\;\;S(p_{0}^{1})=\q^{-s}p_{n-1}^{1},
\end{gather}
for $0\leq l,m,t< \frac{n^{2}}{s}$ and $0\leq i,j,k\leq n-1$, where $\delta_{a,b}$ is the Kroneck notation which equals to $1$ if $a=b$ and $0$ otherwise.
\begin{remark} \emph{To get simplicity, we change the multiplication formula defined in Corollary 3.9 of \cite{HLY} slightly into our formula (2.9). To recover the
original formula given in Corollary 3.9 of \cite{HLY} from (2.9), just substitute $q$ by $\q q$.  }
\end{remark}

\subsection{Quasi-Hopf algebra $A(n,s,q)$.}
  A quasi-bialgebra $(H,\M, \mu, \Delta, \varepsilon, \phi)$ is a
$\k$-algebra $(H,\M,\mu)$ with algebra morphisms $\Delta:\;H\to
H\otimes H$ (the comultiplication) and $\varepsilon:\; H\to
\k$ (the counit), and an invertible element $\phi\in H\otimes
H\otimes H$ (the reassociator), such that
\begin{gather}
(id\otimes \Delta)\Delta(a)\phi=\phi(\Delta\otimes id)\Delta(a),\;\;a\in H,\\
(id\otimes id\otimes\Delta)(\phi)(\Delta\otimes id\otimes id)(\phi)=(1\otimes \phi)
(id\otimes \Delta\otimes id)(\phi)(\phi\otimes 1),\\
(\varepsilon\otimes id)\Delta=id=(id\otimes \varepsilon)\Delta,\\
(id\otimes \varepsilon\otimes)(\phi)=1\otimes 1.
\end{gather}

We denote $\phi=\sum X^{i}\otimes Y^{i}\otimes Z^{i}$ and
$\phi^{-1}=\sum \overline{X}^{i}\otimes \overline{Y}^{i}\otimes
\overline{Z}^{i}$. Then a quasi-bialgebra $H$ is called a
quasi-Hopf algebra if there is a linear algebra antimorphism
$S:\;H\to H$ (the antipode) and elements
$\alpha,\beta\in H$ satisfying for all $a\in H$,
\begin{gather}
\sum S(a_{(1)})\alpha a_{(2)}=\alpha \varepsilon(a),\;\;\sum a_{(1)}\beta S(a_{(2)})=\beta \varepsilon(a),\\
\sum X^{i}\beta S(Y^{i})\alpha Z^{i}=1=\sum S(\overline{X^{i}})\alpha \overline{Y^{i}}\beta S(\overline{Z^{i}}).
\end{gather}

We call an invertible element $J\in H\otimes H$ is a \emph{twist} of
$H$ if it satisfies $(\varepsilon\otimes id)(J)=(id\otimes
\varepsilon)(J)=1$. For a twist $J=\sum f_{i}\otimes g_{i}$ with
inverse $J^{-1}=\sum \overline{f_{i}}\otimes \overline{g_{i}}$, set
$$\alpha_{J}:= \sum S(\overline{f_{i}})\alpha \overline{g_{i}},\;\;
\beta_{J}:= \sum f_{i}\beta S(g_{i}).$$ It is explained that given a
twist $J$ of $H$, if $\beta_{J}$ is invertible then one can
construct a new quasi-Hopf algebra structure
$H_{J}=(H,\Delta_{J},\varepsilon,\Phi_{J},S_{J},\beta_{J}\alpha_{J})$
on the algebra $H$, where
$$\Delta_{J}(a)=J\Delta(a)J^{-1},\;\;a\in H,$$
$$\Phi_{J}=(1\otimes J)(id\otimes \Delta)(J)(\Delta\otimes id)(J^{-1})(J\otimes 1)^{-1}$$
and
$$S_{J}(a)=\beta_{J}S(a)\beta_{J}^{-1},\;\;a\in H.$$

Next we will give the definition of the quasi-Hopf algebras $A(n,s,q)$, which will include the quasi-Hopf algebras $A(q)$ constructed by Gelaki
\cite{Gel} as special examples. The dualities of such quasi-Hopf algebras were constructed in \cite{HLY} and will be recalled in the next subsection.

Let $n$ be a positive integer, $\q$ an $n$-th primitive root of unity and $\k\mathbb{Z}_{n}$ the cyclic group algebra of order $n$. We denote a generator of
$\mathbb{Z}_{n}$ by $g_{2}$ and define
\begin{equation}1_{i}:=\frac{1}{n}\sum_{j=0}^{n-1}(\q^{n-i})^{j}g_{2}^{j}.\end{equation}

 For any $0\leq s\leq n-1$ which is a factor of $n$, i.e., $s|n$,  and $q$ an $n$-th primitive root of $\q$, the quasi-Hopf algebra $A(n,s,q)$ is defined as follows.
As an associative algebra, it is generated by $x,g_{2}$ and satisfies the following relations
\begin{gather}g_{2}^{n}=1,\;\;\;\; x^{\frac{n^{2}}{s}}=0,\;\;\;\;g_{2}x g^{-1}_{2}=\q x.
\end{gather}
The reassociator $\phi_{s}$, the comultiplication $\D$, the counit $\e$, the elements $\alpha,\beta$ and the antipode $S$
are given through
\begin{gather} \phi_{s}=\sum_{i,j,k=0}^{n-1}\q^{si[\frac{j+k}{n}]}1_{i}\otimes
1_{j}\otimes 1_{k},\\
\D(g_{2})=g_{2}\otimes g_{2},\;\;\D(x)=1\otimes\sum_{i=1}^{n-1}1_{i}x+g_{2}^{s}\otimes 1_{0}x+x\otimes \sum_{i=0}^{n-1}q^{-si}1_{i},\\
\alpha=g_{2}^{-s},\;\;\beta=1\\
S(g_{2})=g_{2}^{-1},\;\; S(x)=-x\sum_{i=0}^{n-1}q^{s(i-n)}1_{i}.
\end{gather}

\begin{lemma} The algebra $(A(n,s,q),\M,\mu, \D,\e,\phi_{s},S,\alpha,\beta)$ is a quasi-Hopf algebra and isomorphic to $(M(n,s,q)^{\ast},\D^{\ast},
\e^{\ast},\M^{\ast},\mu^{\ast},\Phi_{s}^{\ast},S^{\ast},\alpha^{\ast},\beta^{\ast})$.
\end{lemma}
\begin{proof} One can show this through direct computations. For our purpose, it is better to establish a direct isomorphism between
$M(n,s,q)^{\ast}$ and $A(n,s,q)$. To attack it, we need give a dual basis of $M(n,s,q)$. Recall $\{p_{i}^{l}|0\leq i\leq n-1, \;0\leq l< n[\frac{n}{s}]\}$ is a basis
of $M(n,s,q)$. Let $\{(p_{i}^{l})^{\ast}|0\leq i\leq n-1, \;0\leq l< n[\frac{n}{s}]\}$ be the canonical dual basis of $M(n,s,q)$.
Define
\begin{equation}\varphi:\;A(n,s,q)\to M(n,s,q)^{\ast},\;\;1_{i}\mapsto (p_{i}^{0})^{\ast},\;x\mapsto \sum_{j=0}^{n-1}(p_{j}^{1})^{\ast},
\end{equation}
where $0\leq i\leq n-1$. It is tedious to show that $\varphi$ gives the desired isomorphism of quasi-Hopf algebras between $M(n,s,q)^{\ast}$ and $A(n,s,q)$.
\end{proof}

\begin{remark} \emph{(1) Take $s=1$ and the resulting quasi-Hopf algebra $A(n,1,q)$ is indeed isomorphic the quasi-Hopf algebra $A(q)$ constructed in \cite{Gel}. In this paper, $A(q)$ is
denoted by $A(n,q)$ for consistence. }

\emph{(2) By the definitions of $\Phi_{s}$ and $\phi_{s}$ defined in (2.7) and (2.20), it is easy to see the assumption that $s$ is factor of $n$ is not restrictive. Thus, throughout of this paper,
    we always take this assumption.}
\end{remark}

\subsection{Drinfeld double of a quasi-Hopf algebra.}
The construction of the Drinfeld double of a quasi-Hopf algebra is not, at least, a trivial generalization from the Hopf to quasi-Hopf
case. After all, the double of a Hopf algebra $H$ is modelled on $H\otimes H^{\ast}$, with $H$ and $H^{\ast}$ as subalgebras. But if $H$ is just a quasi-Hopf algebra, then
$H^{\ast}$ is not an associative algebra, so one is at a loss looking for an associative algebra structure on $H\otimes H^{\ast}$ and even expect that the double
should be some kind of hybrid object. Majid \cite{Maj} settled this problem at first. He gave a conceptual way to show that the double $D(H)$ is still a quasi-Hopf algebra.
Hausser and Nill \cite{FN} gave a computable realization of $D(H)$ on $H\otimes H^{\ast}$. A more explicit version was gotten by Schauenburg \cite{Sch}. Here we will recall the Schauenburg's
construction.

Let $(H,\M, \mu, \Delta, \varepsilon, \phi, S, \alpha, \beta)$ be a finite dimensional quasi-Hopf algebra. Assume $\phi=\phi^{(1)}\otimes
\phi^{(2)}\otimes \phi^{(3)}=\sum X^{i}\otimes Y^{i}\otimes Z^{i}$ and
$\phi^{-1}=\phi^{(-1)}\otimes
\phi^{(-2)}\otimes \phi^{(-3)}=\sum \overline{X}^{i}\otimes \overline{Y}^{i}\otimes
\overline{Z}^{i}$. Define
\begin{gather} \mathbf{\gamma}:=\sum (S(U^{i})\otimes S(T^{i}))(\alpha \otimes \alpha)(V^{i}\otimes W^{i}),\\
               \mathbf{f}:=\sum (S\otimes S)(\D^{op}(\overline{X}^{i}))\cdot \mathbf{\gamma} \cdot \D(\overline{Y}^{i}\beta S(\overline{Z}^{i})),\\
               \mathbf{\chi}:=(\phi \otimes 1)(\D\otimes id\otimes id)(\phi^{-1}),\\
               \mathbf{\omega}:=(1\otimes 1\otimes 1\otimes \tau(\mathbf{f}^{-1}))(id\otimes \D\otimes S\otimes S)(\mathbf{\chi})(\phi\otimes 1\otimes 1),
\end{gather}
where $(1\otimes \phi^{-1})(id\otimes id\otimes \D)(\phi)=\sum T^{i}\otimes U^{i}\otimes V^{i}\otimes W^{i}$ and $\tau$ is the twist, i.e., $\tau(a\otimes b)=b\otimes a$.

As a linear space, $D(H)=H\otimes H^{\ast}$ and we write $h\bowtie \psi:=h\otimes \psi\in D(H)$. There are two canonical actions, denoted by $\rightharpoonup,\;\leftharpoonup$, of $H$
on $H^{\ast}$. By definition, for any $a,b\in H$ and $\psi\in H^{\ast}$
$$\rightharpoonup:\;\;H\otimes H^{\ast}\To H^{\ast}, \;\;\;\;(a\rightharpoonup \psi)(b)=\psi(ba),$$
  $$\leftharpoonup:\;\;H^{\ast}\otimes H\To H^{\ast}, \;\;\;\;( \psi\leftharpoonup a)(b)=\psi(ab).$$
Define a map $\textbf{T}:\;H^{\ast}\to D(H)$ by
\begin{equation}\textbf{T}(\psi)=\phi^{(1)}_{(2)}\bi S(\phi^{(2)})\alpha \phi^{(3)}\rightharpoonup \psi\leftharpoonup \phi^{(1)}_{(1)}.
\end{equation}
With such preparations,  $D(H)$ can be described as the following form (see Theorems 6.3 and 9.3 in \cite{Sch}):

\begin{theorem} Let $H$ be a finite dimensional quasi-Hopf algebra. The quasi-Hopf structure on $D(H)=H\otimes H^{\ast}$, which contains $H$ as a subquasi-Hopf algebra through the embedding $h\hookrightarrow h\bi \e$, is determined by:

\emph{(1)} As an associative algebra, it is generated by $H$ and $\emph{\textbf{T}}(H^{\ast})$ and multiplication rule is
\begin{eqnarray*}&&(g\bi \varphi)(h\bi \psi)\\
&&=gh_{(1)(2)}\omega^{(3)}\bi (\omega^{(5)}\rightharpoonup \psi \leftharpoonup \omega^{(1)})(\omega^{(4)}S(h_{(2)})\rightharpoonup \varphi\leftharpoonup
h_{(1)(1)}\omega^{(2)}),\;\;\;\;(\star)
\end{eqnarray*}
and as a quasi-coalgebra, the comultiplication is given by
\begin{eqnarray*}\D_{D}(\mathbf{T}(\psi))&=&
 \tilde{\phi}^{(2)}\mathbf{T}(\psi_{(1)}\leftharpoonup \tilde{\phi}^{(1)})\phi^{(-1)}\phi^{(1)}\\
 &\otimes& \tilde{\phi}^{(3)}
\phi^{(-3)}\mathbf{T}(\phi^{(3)}\rightharpoonup \psi_{(2)}\leftharpoonup \phi^{-2})\phi^{(2)},\;\;\;\;(\star\star)\end{eqnarray*}
for $g,h\in H$ and $\varphi,\psi\in H^{\ast}$, where $\tilde{\phi}$ denote another copy of $\phi$.

\emph{(2)} The reassociator $\phi_{D}$,  the counit $\e_{D}$, elements $\alpha_{D},\beta_{D}$ and the antipode $S_{D}$ are given by
\begin{gather} \phi_{D}=\phi\bi \e,
\e_{D}(\mathbf{T}(\psi))=\psi(\phi^{(1)}S(\phi^{(2)})\alpha \phi^{(3)}),\\
\alpha_{D}=\alpha\bi \e,\;\;\;\;\beta_{D}=\beta\bi \e,\\
S_{D}(\mathbf{T}(\psi))=\mathbf{f}^{(2)}\mathbf{T}(\mathbf{f}^{(-2)}\rightharpoonup S^{-1}(\psi)\leftharpoonup \mathbf{f}^{(1)})\mathbf{f}^{(-1)},
\end{gather}
for $\psi\in H^{\ast}$.
\end{theorem}

\begin{remark} \emph{(1) By formula $(\star)$, $1\bi \e$ is the unit element of $D(H)$. Moreover, as a special case of this formula, we also have}
\begin{equation}(1\bi \varphi)(h\bi \e)=h_{(1)(2)}\bi S(h_{(2)})\rightharpoonup \varphi\leftharpoonup h_{(1)(1)},
\end{equation}
\emph{for $h\in H$ and $\varphi\in H^{\ast}$.}

 \emph{(2)  In the process of our computations, we find that there are some misprints in \cite{Sch} and \cite{FN}. Especially, there are misprints in the expression of the element $\mathbf{f}$ given both in \cite{Sch} and \cite{FN}, the element $\chi$ given in \cite{Sch} and the comultiplication formula given in \cite{Sch}. The correct versions are (2.26), (2.27) and ($\star\star$).}
\end{remark}

\subsection{The quasi-Hopf algebra $\Q_{s}\mathbf{u}_{q}(\mathfrak{sl}_{2})$. }

    The quasi-Hopf algebra $\Q_{s}\mathbf{u}_{q}(\mathfrak{sl}_{2})$ is defined as follows. As an associative algebra, it is generated by four
    elements $g_{1},g_{2},x,y$ satisfying
    \begin{gather} g_{1}^{n}=g_{2}^{2s},\;\;g_{2}^{n}=1, \;\;g_{1}g_{2}=g_{2}g_{1},\;\;x^{\frac{n^{2}}{s}}=y^{\frac{n^{2}}{s}}=0,\\
    g_{1}x g_{1}^{-1}=\q^{-s}q^{2s}x,\;\;\;\;g_{2}x g_{2}^{-1}=\q x,\\
    g_{1}y g_{1}^{-1}=\q^{s}q^{-2s}y,\;\;\;\;g_{2}y g_{2}^{-1}=\q^{-1} y,\\
    yx-q^{s}xy=1-g_{1}g_{2}^{s}.
    \end{gather}
    Define
    \begin{equation}1_{i}:=\frac{1}{n}\sum_{j=0}^{n-1}(\q^{n-i})^{j}g_{2}^{j}.\end{equation}
    The reassociator $\phi_{s}$, the comultiplication $\D$, the counit $\e$, the elements $\alpha,\beta$ and the antipode $S$
    are given through
    \begin{gather} \phi_{s}=\sum_{i,j,k=0}^{n-1}\q^{si[\frac{j+k}{n}]}1_{i}\otimes
    1_{j}\otimes 1_{k},\\
    \D(g_{1})=g_{1}\otimes g_{1},\;\;\;\;\D(g_{2})=g_{2}\otimes g_{2},\\
    \D(x)=1\otimes\sum_{i=1}^{n-1}1_{i}x+g_{2}^{s}\otimes 1_{0}x+x\otimes \sum_{i=0}^{n-1}q^{-si}1_{i},\\
    \D(y)=y\otimes\sum_{i=0}^{n-1}q^{si}1_{i}+g_{1}g_{2}^{s}\otimes y\sum_{i=1}^{n-1}1_{i}+g_{1}\otimes y1_{0},\\
    \e(g_{1})=\e(g_{2})=1,\;\;\;\;\e(x)=\e(y)=0,\\
    \alpha=g_{2}^{-s},\;\;\;\;\beta=1\\
    S(g_{1})=g_{1}^{-1},\;\;\;\;S(g_{2})=g_{2}^{-1},\\
    S(x)=-x\sum_{i=0}^{n-1}q^{s(i-n)}1_{i},\;\;S(y)=-g_{1}^{-1}g_{2}^{-s}y\sum_{i=0}^{n-1}q^{s(n-i)'}1_{i}.
    \end{gather} where for any integer $i\in \mathbb{N}$, we denote by $i'$ the remainder of division of $i$ by $n$.

\begin{lemma} $\Q_{s}\mathbf{u}_{q}(\mathfrak{sl}_{2})$ is a quasi-Hopf algebra.
\end{lemma}
 \begin{proof} We will show that $D(A(n,s,q))$ is isomorphic to $\Q_{s}\mathbf{u}_{q}(\mathfrak{sl}_{2})$ and thus $\Q_{s}\mathbf{u}_{q}(\mathfrak{sl}_{2})$ is a quasi-Hopf
 algebra. Moreover, it is quasitriangular. Of course, one can show the result by direct computations. Here we only check the equality $\D(y)\D(x)-q^{s}\D(x)\D(y)=\D(1)-\D(g_{1})\D(g_{2}^{s})$.
 Indeed,
 \begin{eqnarray*}\D(y)\D(x)&=&(y\otimes\sum_{i=0}^{n-1}q^{si}1_{i}+g_{1}g_{2}^{s}\otimes y\sum_{i=1}^{n-1}1_{i}+g_{1}\otimes y1_{0})\cdot\\
 &&(1\otimes\sum_{i=1}^{n-1}1_{i}x+g_{2}^{s}\otimes 1_{0}x+x\otimes \sum_{i=0}^{n-1}q^{-si}1_{i})\\
 &=&y\otimes \sum_{i=1}^{n-1}q^{si}1_{i}x+yg_{2}^{s}\otimes 1_{0}x+yx\otimes 1+ g_{1}g_{2}^{s}\otimes y\sum_{i=1}^{n-1}1_{i}x\\
 &&+g_{1}g_{2}^{s} x\otimes y\sum_{i=1}^{n-1}q^{-si}1_{i}+g_{1}g_{2}^{s}\otimes y1_{0}x+g_{1}x\otimes y1_{0}\\
 &=&y\otimes \sum_{i=1}^{n-1}q^{si}1_{i}x+yg_{2}^{s}\otimes 1_{0}x+yx\otimes 1+g_{1}g_{2}^{s}\otimes yx\\
 &&+g_{1}g_{2}^{s} x\otimes y\sum_{i=1}^{n-1}q^{-si}1_{i}+g_{1}x\otimes y1_{0},
  \end{eqnarray*}
 and
 \begin{eqnarray*}q^{s}\D(x)\D(y)&=&q^{s}(1\otimes\sum_{i=1}^{n-1}1_{i}x+g_{2}^{s}\otimes 1_{0}x+x\otimes \sum_{i=0}^{n-1}q^{-si}1_{i})\cdot\\
 &&(y\otimes\sum_{i=0}^{n-1}q^{si}1_{i}+g_{1}g_{2}^{s}\otimes y\sum_{i=1}^{n-1}1_{i}+g_{1}\otimes y1_{0})\\
 &=&q^{s}[y\otimes \sum_{i=1}^{n-1}q^{s(i-1)}1_{i}x+g_{1}g_{2}^{s}\otimes xy\sum_{i=1}^{n-1}1_{i}+g_{2}^{s}y\otimes q^{s(n-1)}1_{0}x\\
 &&+g_{1}g_{2}^{s}\otimes xy1_{0}+xy\otimes 1+xg_{1}g_{2}^{s}\otimes y\sum_{i=1}^{n-1}q^{-s(i-1)}1_{i}\\
 &&+ xg_{1}\otimes q^{-s(n-1)}y1_{0}]\\
 &=&q^{s}[y\otimes \sum_{i=1}^{n-1}q^{s(i-1)}1_{i}x+g_{1}g_{2}^{s}\otimes xy+q^{-s}yg_{2}^{s}\otimes 1_{0}x\\
 && +xy\otimes 1+q^{-s}g_{1}g_{2}^{s}x\otimes y\sum_{i=1}^{n-1}q^{-si}1_{i}+q^{-s}g_{1}x\otimes y1_{0}].
  \end{eqnarray*}
Therefore,
  \begin{eqnarray*}\D(y)\D(x)-q^{s}\D(x)\D(y)&=&(yx-q^{s}xy)\otimes 1 + g_{1}g_{2}^{s}\otimes (yx-q^{s}xy)\\
  &=&(1-g_{1}g_{2}^{s})\otimes 1+g_{1}g_{2}^{s}\otimes (1-g_{1}g_{2}^{s})\\
  &=&1\otimes 1 -g_{1}g_{2}^{s} \otimes g_{1}g_{2}^{s}\\
  &=&\D(1)-\D(g_{1})\D(g_{2}^{s}).
  \end{eqnarray*}
 \end{proof}

\section{The Drinfeld double of $A(n,s,q)$}

The main result of this section is the following result.
\begin{theorem} The Drinfeld double $D(A(n,s,q))$ of $A(n,s,q)$ is isomorphic to the quasi-Hopf algebra $\Q_{s}\mathbf{u}_{q}(\mathfrak{sl}_{2})$. That is,
$$D(A(n,s,q))\cong \Q_{s}\mathbf{u}_{q}(\mathfrak{sl}_{2}).$$
\end{theorem}

To show it, we need understand $D(A(n,s,q))$ well.
Recall that for any integer $i\in \mathbb{N}$, we denote by $i'$ the remainder of division of $i$ by $n$. The following
lemma is useful in our computations.

\begin{lemma} For any two natural numbers $i,j$, we always have
\begin{equation}[\frac{i+j'}{n}]=[\frac{i+j}{n}]-[\frac{j}{n}].
\end{equation}\end{lemma}

\begin{proof}\begin{eqnarray*}[\frac{i+j'}{n}]&=&[\frac{i+j-[\frac{j}{n}]n}{n}]\\
&=&[\frac{i+j}{n}]-[\frac{j}{n}].
\end{eqnarray*}
\end{proof}

The formula (3.1) will be used frequently without explanation. Recall the reassociator of $A(n,s,q)$ is defined to be
$$ \phi_{s}=\sum_{i,j,k=0}^{n-1}\q^{si[\frac{j+k}{n}]}1_{i}\otimes
1_{j}\otimes 1_{k}.$$ The next lemma will give the explicit formalism of the elements
$\gamma, \;\mathbf{f},\;\chi$ and $\omega$ in such case. Throughout this section, $\phi_{s}$ is denoted
by $\phi$ for short when this is no confusion.

\begin{lemma} For the quasi-Hopf algebra $A(n,s,q)$, we have
\begin{gather} \gamma=\sum_{j,k=0}^{n-1}\q^{s(j+k)[\frac{j+k}{n}]+sk[\frac{n-j}{n}]-s(j+2k)}1_{j}\otimes 1_{k},\\
\mathbf{f}=\sum_{j,k=0}^{n-1}\q^{s(j+k)[\frac{j+k}{n}]+sk[\frac{n-j}{n}]-sk}1_{j}\otimes 1_{k},\\
\chi=\sum_{i_{1},i_{2},j,k=0}^{n-1}\q^{si_{1}[\frac{i_{2}+j}{n}]-s(i_{1}+i_{2})[\frac{j+k}{n}]}1_{i_{1}}\otimes 1_{i_{2}}\otimes 1_{j}\otimes 1_{k},\end{gather}
and
\begin{eqnarray*}\omega&=&\sum_{i_{1},i_{2},i_{3},i_{4},i_{5}=0}^{n-1}\q^{si_{5}-s(\Sigma_{t=1}^{5}i_{t})[\frac{i_{4}+i_{5}}{n}]
+s i_{1}[\frac{i_{2}+i_{3}+i_{4}}{n}]-si_{5}[\frac{n-i_{4}}{n}]}\\
&&\;\;\;\;\;\;\;\;\;\;1_{i_{1}}\otimes 1_{i_{2}}\otimes 1_{i_{3}}\otimes S(1_{i_{4}})\otimes S(1_{i_{5}}).
\end{eqnarray*}
\end{lemma}
\begin{proof}At first, we have
\begin{eqnarray*}T^{i}\otimes U^{i}\otimes V^{i}\otimes W^{i}:&=&(1\otimes \phi^{-1})(id\otimes id\otimes\D)(\phi)\\
&=& \sum_{i,j,k=0}^{n-1}\q^{-si[\frac{j+k}{n}]}1\otimes 1_{i}\otimes
1_{j}\otimes 1_{k}\cdot\\&&\sum_{i_{1},j_{1},k_{1},k_{2}=0}^{n-1}
\q^{si_{1}[\frac{j_{1}+(k_{1}+k_{2})'}{n}]}1_{i_{1}}\otimes
1_{j_{1}}\otimes 1_{k_{1}}\otimes 1_{k_{2}}\\
&=&\sum_{i_{1},i,j,k=0}^{n-1}
\q^{-si[\frac{j+k}{n}]}\q^{si_{1}[\frac{i+(j+k)'}{n}]}1_{i_{1}}\otimes
1_{i}\otimes 1_{j}\otimes 1_{k},
\end{eqnarray*}
and
\begin{eqnarray*}\gamma&=&(S(U^{i})\otimes S(T^{i}))(\alpha\otimes\alpha)(V^{i}\otimes W^{i})\\
&=&\sum_{i_{1},i,j,k=0}^{n-1}
\q^{-si[\frac{j+k}{n}]}\q^{si_{1}[\frac{i+(j+k)'}{n}]}S(1_{i})g_{2}^{-s}1_{j}\otimes S(1_{i_{1}})g_{2}^{-s}1_{k}\\
&=&\sum_{i_{1},i,j,k=0}^{n-1}
\q^{-si[\frac{j+k}{n}]}\q^{si_{1}[\frac{i+(j+k)'}{n}]}\q^{-s(j+k)}S(1_{i})1_{j}\otimes S(1_{i_{1}})1_{k}\\
&=&\sum_{j,k=0}^{n-1}
\q^{-s(n-j)'[\frac{j+k}{n}]}\q^{s(n-k)'[\frac{(n-j)'+(j+k)'}{n}]}\q^{-s(j+k)}1_{j}\otimes 1_{k}\\
&=&\sum_{j,k=0}^{n-1}
\q^{sj[\frac{j+k}{n}]-sk[\frac{n+k}{n}]+sk[\frac{n-j}{n}]+sk[\frac{j+k}{n}]-s(j+k)}1_{j}\otimes 1_{k}\\
&=&\sum_{j,k=0}^{n-1}\q^{s(j+k)[\frac{j+k}{n}]+sk[\frac{n-j}{n}]-s(j+2k)}1_{j}\otimes 1_{k}.
\end{eqnarray*}
Therefore,
\begin{eqnarray*}\mathbf{f}&=&\sum (S\otimes S)(\D^{op}(\overline{X}^{i}))\cdot \mathbf{\gamma} \cdot \D(\overline{Y}^{i}\beta S(\overline{Z}^{i}))\\
&=&\sum_{i_{1},i_{2},j_{1},k_{1}=0}^{n-1}\q^{-s(i_{1}+i_{2})'[\frac{j_{1}+k_{1}}{n}]}(S(1_{i_{1}})\otimes S(1_{i_{2}}))\gamma\D(1_{j_{1}}1_{(n-k_{1})'})\\
&=&\sum_{j,k=0}^{n-1}\q^{-s((n-j)'+(n-k)')'[\frac{(j+k)'+(n-(j+k)')'}{n}]}\gamma 1_{j}\otimes 1_{k}\\
&=& \sum_{j,k=0}^{n-1}\q^{s(j+k)}\gamma 1_{j}\otimes 1_{k}\\
&=&\sum_{j,k=0}^{n-1}\q^{s(j+k)[\frac{j+k}{n}]+sk[\frac{n-j}{n}]-sk}1_{j}\otimes 1_{k}.
\end{eqnarray*}
The computation for $\chi$ is easy. Indeed,
\begin{eqnarray*}\chi&=&(\phi\otimes 1)(\D\otimes id\otimes id)(\phi^{-1})\\
&=&\sum_{i,j,k=0}^{n-1}\q^{si[\frac{j+k}{n}]}1_{i}\otimes1_{j}\otimes1_{k}\otimes1\\
&& \sum_{i_{1},i_{2},j_{1},k_{1}=0}^{n-1}\q^{-s(i_{1}+i_{2})'[\frac{j_{1}+k_{1}}{n}]}1_{i_{1}}\otimes1_{i_{2}}\otimes1_{j_{1}}\otimes1_{k_{1}}\\
&=& \sum_{i_{1},i_{2},j,k=0}^{n-1}\q^{si_{1}[\frac{i_{2}+j}{n}]-s(i_{1}+i_{2})[\frac{j+k}{n}]}1_{i_{1}}\otimes 1_{i_{2}}\otimes 1_{j}\otimes 1_{k}.
\end{eqnarray*}
The desire element $\omega$ can be gotten in the following way:
\begin{eqnarray*} \omega&=&(1\otimes 1\otimes 1\otimes \tau(\mathbf{f}^{-1}))(id\otimes \D\otimes S\otimes S)(\mathbf{\chi})(\phi\otimes 1\otimes 1)\\
&=& \sum_{j,k=0}^{n-1}\q^{sk-s(j+k)[\frac{j+k}{k}]-sk[\frac{n-j}{n}]}1\otimes 1\otimes1\otimes 1_{k}\otimes 1_{j}\sum_{i_{1},\ldots,i_{5}=0}^{n-1}\\
&& \q^{si_{1}[\frac{(i_{2}+i_{3})'+i_{4}}{n}]-s(i_{1}+(i_{2}+i_{3})')[\frac{i_{4}+i_{5}}{n}]}1_{i_{1}}
\otimes 1_{i_{2}}\otimes 1_{i_{3}}\otimes S(1_{i_{4}})\otimes S(1_{i_{5}}) \\
&&\sum_{i_{1},i_{2},i_{3}=0}^{n-1}\q^{si_{1}[\frac{i_{2}+i_{3}}{n}]}1_{i_{1}}
\otimes 1_{i_{2}}\otimes 1_{i_{3}}\otimes1\otimes 1\\
&=& \sum_{i_{1},\ldots,i_{5}=0}^{n-1}\q^{si_{1}[\frac{i_{2}+i_{3}}{n}]+s(n-i_{4})-s(2n-i_{4}-i_{5})[\frac{(n-i_{4})'+(n-i_{5})'}{n}]
-s(n-i_{4})[\frac{n-(n-i_{5})'}{n}]}\\
&&\q^{si_{1}[\frac{i_{2}+i_{3}+i_{4}}{n}]-si_{1}[\frac{i_{2}+i_{3}}{n}]-s(i_{1}+i_{2}+i_{3})[\frac{i_{4}+i_{5}}{n}]}1_{i_{1}}
\otimes 1_{i_{2}}\otimes 1_{i_{3}}\otimes S(1_{i_{4}})\otimes S(1_{i_{5}})\\
&=&\sum_{i_{1},\ldots,i_{5}=0}^{n-1}\q^{-si_{4}+s(i_{4}+i_{5})[\frac{(n-i_{4})'+(n-i_{5})'}{n}]+si_{4}[\frac{n-i_{5}}{n}]
+si_{1}[\frac{i_{2}+i_{3}+i_{4}}{n}]-s(i_{1}+i_{2}+i_{3})[\frac{i_{4}+i_{5}}{n}]}\\
&&\;\;\;\;\;\;\;\;\;\;1_{i_{1}}\otimes 1_{i_{2}}\otimes 1_{i_{3}}\otimes S(1_{i_{4}})\otimes S(1_{i_{5}})\\
&=&\sum_{i_{1},\ldots,i_{5}=0}^{n-1}\q^{-si_{4}+s(i_{4}+i_{5})[\frac{2n-i_{4}-i_{5}}{n}]-si_{5}[\frac{n-i_{4}}{n}]
+si_{1}[\frac{i_{2}+i_{3}+i_{4}}{n}]-s(i_{1}+i_{2}+i_{3})[\frac{i_{4}+i_{5}}{n}]}\\
&&\;\;\;\;\;\;\;\;\;\;1_{i_{1}}\otimes 1_{i_{2}}\otimes 1_{i_{3}}\otimes S(1_{i_{4}})\otimes S(1_{i_{5}})\\
&=&\sum_{i_{1},\ldots,i_{5}=0}^{n-1}\q^{-si_{4}+s(i_{4}+i_{5})(1-[\frac{i_{4}+i_{5}}{n}])-si_{5}[\frac{n-i_{4}}{n}]
+si_{1}[\frac{i_{2}+i_{3}+i_{4}}{n}]-s(i_{1}+i_{2}+i_{3})[\frac{i_{4}+i_{5}}{n}]}\\
&&\;\;\;\;\;\;\;\;\;\;1_{i_{1}}\otimes 1_{i_{2}}\otimes 1_{i_{3}}\otimes S(1_{i_{4}})\otimes S(1_{i_{5}})\\
&=&\sum_{i_{1},i_{2},i_{3},i_{4},i_{5}=0}^{n-1}\q^{si_{5}-s(\Sigma_{t=1}^{5}i_{t})[\frac{i_{4}+i_{5}}{n}]
+s i_{1}[\frac{i_{2}+i_{3}+i_{4}}{n}]-si_{5}[\frac{n-i_{4}}{n}]}\\
&&\;\;\;\;\;\;\;\;\;\;1_{i_{1}}\otimes 1_{i_{2}}\otimes 1_{i_{3}}\otimes S(1_{i_{4}})\otimes S(1_{i_{5}}).
\end{eqnarray*}
\end{proof}

Once the element $\omega$ is known, the multiplication rules of $D(A(n,s,q))$ can be determined according to formula $(\star)$.

\begin{proposition} In $D(A(n,s,q))$, we have the following relations
\begin{eqnarray}&&(g_{2}\bi \e)^{n}=1\bi \e,\;\;(x\bi \e)^{\frac{n^{2}}{s}}=0,\\
&&(g_{2}\bi \e)(x\bi \e)(g_{2}\bi \e)^{-1}=\q (x\bi \e),\\
&&(1\bi g)(g_{2}\bi \e)=(g_{2}\bi \e)(1\bi g),\;\;(\sum_{i=0}^{n-1}q^{si}1_{i}\bi g)^{n}=g_{2}^{2s}\bi \e,\\
&&(1\bi p_{0}^{1})^{\frac{n^{2}}{s}}=0,\;\;(g_{2}\bi \e)(1\bi p_{0}^{1})(g_{2}\bi \e)^{-1}=\q^{-1}(1\bi p_{0}^{1}),\\
&&(\sum_{i=0}^{n-1}q^{si}1_{i}\bi g)(x\bi \e)(\sum_{i=0}^{n-1}q^{si}1_{i}\bi g)^{-1}=\q^{-s}q^{2s}(x\bi \e),\\
&&(\sum_{i=0}^{n-1}q^{si}1_{i}\bi g)(1\bi p_{0}^{1})(\sum_{i=0}^{n-1}q^{si}1_{i}\bi g)^{-1}=\q^{s} q^{-2s}(1\bi p_{0}^{1}),
\end{eqnarray}
and
\begin{eqnarray*}(\sum_{i=0}^{n-1}q^{si}1_{i}\bi p_{0}^{1})(x\bi \e)&-&q^{s}(x\bi \e)(\sum_{i=0}^{n-1}q^{si}1_{i}\bi p_{0}^{1})\\
&=&(1\bi \e)-(g_{2}^{s}\sum_{i=0}^{n-1}q^{si}1_{i}\bi g).
\end{eqnarray*}
\end{proposition}
\begin{proof} Formulas (3.5) and (3.6) are clear since $A(n,s,q)$ is a subquasi-Hopf algebra of its double. The first part of formula (3.7) and the second part of formula (3.8) are direct consequences of formula (2.33). For the second part of (3.7), one has
\begin{eqnarray*}(1\bi g)(1\bi g)&=&\omega^{(3)}\bi (\omega^{(5)}\rightharpoonup g\leftharpoonup \omega^{(1)})(\omega^{(4)}\rightharpoonup g\leftharpoonup \omega^{(2)})\\
&=&\sum_{i_{3}=0}^{n-1}\q^{-s[\frac{(n-1)'+(n-1)'}{n}]}1_{i_{3}}\bi g^{2}=g_{2}^{-s}\bi g^{2}.
\end{eqnarray*}
Inductively, one has
\begin{equation} (1\bi g)^{i}=g_{2}^{-s(i-1)}\bi g^{i}
\end{equation}
for $1\leq i\leq n$. Thus $(\sum_{i=0}^{n-1}q^{si}1_{i}\bi g)^{n}=(\sum_{i=0}^{n-1}q^{si}1_{i}\bi \e)^{n}(1\bi g)^{n}=(\sum_{i=0}^{n-1}\q^{si}1_{i}\bi \e)(1\bi g)^{n}=
g_{2}^{2s}\bi \e$.

Since the multiplication of $M(n,s,q)$ is not associative in general, we need two notions. For any algebra (maybe not associative) $A$, let $X\in A$. Define
$$X^{\stackrel{\rightharpoonup}{l}}=:\stackrel{l}{\overbrace{(\cdots(X\cdot X)\cdot X)\cdots)}},\;\;X^{\stackrel{\leftharpoonup}{l}}=:\stackrel{l}{\overbrace{(\cdots(X\cdot(X\cdot X))
\cdots)}}.$$  For the first part of (3.8), we have
\begin{eqnarray*} (1\bi p_{0}^{1})(1\bi p_{0}^{1})&=& \omega^{(3)}\bi (\omega^{(5)}\rightharpoonup p_{0}^{1}\leftharpoonup \omega^{(1)})(\omega^{(4)}\rightharpoonup p_{0}^{1}\leftharpoonup \omega^{(2)})\\
&=&\sum_{i_{3}=0}^{n-1}\q^{s[\frac{1+i_{3}}{n}]}1_{i_{3}}\bi (p_{0}^{1})^{\stackrel{\rightharpoonup}{2}}\\
&=&\sum_{i_{3}=0}^{n-1}\q^{s[\frac{1+i_{3}}{n}]}1_{i_{3}}\bi (p_{0}^{1})^{\stackrel{\leftharpoonup}{2}} .
\end{eqnarray*}
Inductively, one can prove the following equalities:
\begin{gather} (1\bi p_{0}^{1})^{\stackrel{\rightharpoonup}{l}}=g_{2}^{s[\frac{l}{n}]}\sum_{i_{3}=0}^{n-1}\q^{s[\frac{1+i_{3}}{n}]+s[\frac{2+i_{3}}{n}]+\cdots +s[\frac{l'-1+i_{3}}{n}]}
  1_{i_{3}}\bi (p_{0}^{1})^{\stackrel{\leftharpoonup}{l}}, \\
  (1\bi p_{0}^{1})^{\stackrel{\leftharpoonup}{l}}=
  \q^{sl'[\frac{l}{n}]}g_{2}^{s[\frac{l}{n}]}\sum_{i_{3}=0}^{n-1}\q^{s[\frac{1+i_{3}}{n}]+s[\frac{2+i_{3}}{n}]+\cdots +s[\frac{l'-1+i_{3}}{n}]}1_{i_{3}}\bi (p_{0}^{1})^{\stackrel{\rightharpoonup}{l}}.
\end{gather}
 Comparing with Lemma 3.6 in \cite{HLY}, we indeed have $(1\bi p_{0}^{1})^{\stackrel{\rightharpoonup}{l}}=(1\bi p_{0}^{1})^{\stackrel{\leftharpoonup}{l}}$ and $(1\bi p_{0}^{1})^{n[\frac{n}{s}]}=0$.
 Now let's prove the formula (3.9).
\begin{eqnarray*}&&(\sum_{i=0}^{n-1}q^{si}1_{i}\bi g)(x\bi \e)(\sum_{i=0}^{n-1}q^{si}1_{i}\bi g)^{-1}\\ &=&(\sum_{i=0}^{n-1}q^{si}1_{i}x_{(1)(2)}\omega^{(3)}\bi (\omega^{(5)}\rightharpoonup \e\leftharpoonup \omega^{(1)})(\omega^{(4)}S(x_{(2)})\rightharpoonup g\leftharpoonup
 x_{(1)(1)}\omega^{(2)}))\\&&(\sum_{i=0}^{n-1}q^{si}1_{i}\bi g)^{-1}\\
 &=& (\sum_{i=0}^{n-1}q^{si}1_{i}\bi \e)(\sum_{j=1}^{n-1}1_{i}x\bi q^{-(n-1)s}g+ 1_{0}x\bi q^{s}g)(\sum_{i=0}^{n-1}q^{si}1_{i}\bi g)^{-1}\\
 &=&(x\sum_{j=0}^{n-2}q^{s(j+1)}1_{j}\bi q^{-(n-1)s}g+ x1_{n-1}\bi q^{s}g)(\sum_{i=0}^{n-1}q^{si}1_{i}\bi g)^{-1}\\
 &=& (x\sum_{j=0}^{n-2}\q^{-s}q^{2s}q^{sj}1_{j}\bi g+  x\q^{-s}q^{2s}q^{(n-1)s}1_{n-1}\bi g)(\sum_{i=0}^{n-1}q^{si}1_{i}\bi g)^{-1} \\
 &=&\q^{-s}q^{2s}(x\bi \e)(\sum_{i=0}^{n-1}q^{si}1_{i}\bi g)(\sum_{i=0}^{n-1}q^{si}1_{i}\bi g)^{-1} \\
  &=&\q^{-s}q^{2s}(x\bi \e).
 \end{eqnarray*}
 For (3.10), note that $(\sum_{i=0}^{n-1}q^{si}1_{i}\bi g)^{-1}=(g_{2}^{s}\sum_{i=0}^{n-1}q^{-si}1_{i}\bi g^{n-1})$.
\begin{eqnarray*}&&(\sum_{i=0}^{n-1}q^{si}1_{i}\bi g)(1\bi p_{0}^{1})(\sum_{i=0}^{n-1}q^{si}1_{i}\bi g)^{-1}\\
 &=&(\sum_{i=0}^{n-1}q^{si}1_{i}\omega^{(3)}\bi (\omega^{(5)}\rightharpoonup p_{0}^{1}\leftharpoonup \omega^{(1)})(\omega^{(4)}\rightharpoonup g\leftharpoonup \omega^{(2)}))(\sum_{i=0}^{n-1}q^{si}1_{i}\bi g)^{-1}\\
 &=& (\sum_{i=0}^{n-1}\q^{s}q^{si}1_{i}\bi (p_{0}^{1}g))(g_{2}^{s}\sum_{i=0}^{n-1}q^{-si}1_{i}\bi g^{n-1})\\
 &=&(\sum_{i=0}^{n-1}\q^{s}q^{si}1_{i}\bi \e)(\sum_{i_{3}=0}^{n-1}\q^{s}\q^{s-s(1+i_{3})+s(n-1)[\frac{n+1+i_{3}}{n}]+s[\frac{2}{n}]}q^{-s(1+i_{3})'}1_{i_{3}}g_{2}^{s}\bi g^{n-1}(p_{0}^{1}g))\\
 &=&(\sum_{i=0}^{n-1}\q^{s}q^{si}1_{i}\bi \e)(\sum_{i_{3}=0}^{n-1}\q^{s}\q^{-s(1+i_{3})+s[\frac{2}{n}]}q^{-s(1+i_{3})}1_{i_{3}}g_{2}^{s}\bi g^{n-1}(p_{0}^{1}g))  \\
 &=&(\sum_{i=0}^{n-1}\q^{s}q^{si}1_{i}\bi \e)(\sum_{i_{3}=0}^{n-1}\q^{s[\frac{2}{n}]}q^{-s(1+i_{3})}1_{i_{3}}\bi g^{n-1}(p_{0}^{1}g)) \\
  &=&\sum_{i=0}^{n-1}\q^{s}q^{-s}\q^{s[\frac{2}{n}]}1_{i}\bi g^{n-1}(p_{0}^{1}g)\\
  &=&\sum_{i=0}^{n-1}\q^{s}q^{-s}\q^{s[\frac{2}{n}]}1_{i}\bi q^{-s}\q^{-s[\frac{2}{n}]}p_{0}^{1}\\
  &=&\q^{s}q^{-2s} 1\bi p_{0}^{1},
 \end{eqnarray*}
 where the second last equality are gotten from (2.9).

 Now the only task is to prove the last formula in this proposition. Note that
 \begin{eqnarray*}(\D\otimes id)\D(x)&=&1\otimes 1 \otimes \sum_{i=1}^{n-1}1_{i}x+g_{2}^{s}\otimes g_{2}^{s}\otimes 1_{0}x+1\otimes \sum_{i=1}^{n-1}1_{i}x\otimes\sum_{i=0}^{n-1}q^{-si}1_{i}\\
 &&+ g_{2}^{s}\otimes 1_{0}x\otimes  \sum_{i=0}^{n-1}q^{-si}1_{i}+x\otimes \sum_{i=0}^{n-1}q^{-si}1_{i}\otimes\sum_{i=0}^{n-1}q^{-si}1_{i}.
 \end{eqnarray*}
 By applying above five items $1\otimes 1 \otimes \sum_{i=1}^{n-1}1_{i}x,\; g_{2}^{s}\otimes g_{2}^{s}\otimes 1_{0}x,\;1\otimes \sum_{i=1}^{n-1}1_{i}x\otimes\sum_{i=0}^{n-1}q^{-si}1_{i},\;
 g_{2}^{s}\otimes 1_{0}x\otimes  \sum_{i=0}^{n-1}q^{-si}1_{i}$ and $x\otimes \sum_{i=0}^{n-1}q^{-si}1_{i}\otimes\sum_{i=0}^{n-1}q^{-si}1_{i}$ into the following formula
 $$(1\bi p_{0}^{1})(x\bi \e)=x_{(1)(2)}\bi S(x_{2})\rightharpoonup p_{0}^{1}\leftharpoonup x_{(1)(1)},$$
 we get
 $$(1\bi p_{0}^{1})(x\bi \e)=\sum_{i=1}^{n-1}1_{i}x\bi p_{0}^{1}+\q^{s}1_{0}x\bi p_{0}^{1}-g_{2}^{s}\bi g+\sum_{i=0}^{n-1}q^{-si}1_{i}\bi \e.$$
 Multiplying the element $\sum_{i=0}^{n-1}q^{si}1_{i}\bi \e$ to above equality, we have
  \begin{eqnarray*}(\sum_{i=0}^{n-1}q^{si}1_{i}\bi p_{0}^{1})(x\bi \e)&=&q^{s}(x\bi \e)(\sum_{i=0}^{n-1}q^{si}1_{i}\bi p_{0}^{1})\\
  &&+(1\bi \e)-(g_{2}^{s}\sum_{i=0}^{n-1}q^{si}1_{i}\bi g).
    \end{eqnarray*}
  \end{proof}

Next, let us determine the comultiplication, the counit, the antipode and the elements $\alpha,\beta $ of $D(A(n,s,q))$.
From now on, sometimes, we denote $h\bi \varphi\in D(A(n,s,q))$ by $h\varphi$ for short.
\begin{proposition}In $D(A(n,s,q))$, we have
     \begin{eqnarray} && \;\;\;\;\;\D(g_{2})=g_{2}\otimes g_{2},\\
      && \;\;\;\;\;\D(x)=1\otimes \sum_{i=1}^{n-1}1_{i}x+g_{2}^{s}\otimes 1_{0}x+x\otimes \sum_{i=0}^{n-1}q^{-si}1_{i},\\
     && \;\;\;\;\;\D(\sum_{i=0}^{n-1}q^{si}1_{i} g)=\sum_{i=0}^{n-1}q^{si}1_{i} g \otimes \sum_{i=0}^{n-1}q^{si}1_{i} g, \\
      && \;\;\;\;\;\D(\sum_{i=0}^{n-1}q^{si}1_{i}p_{0}^{1})=\sum_{i=0}^{n-1}q^{si}1_{i}p_{0}^{1}\otimes \sum_{i=0}^{n-1}q^{si}1_{i}+g_{2}^{s}\sum_{i=0}^{n-1}q^{si}1_{i} g\otimes \sum_{i=0}^{n-2}q^{si}1_{i}p_{0}^{1}
     \\&& \;\;\;\;\; \;\;\;\;\; \;\;\;\;\; \;\;\;\;\; \;\;\;\;\; \;\;\;\;\;\;\;+\sum_{i=0}^{n-1}q^{si}1_{i} g\otimes q^{s(n-1)}1_{n-1}p_{0}^{1},\notag\\
     && \;\;\;\;\;\e(g_{2})=1,\;\;\;\;\e(x)=0,\\
    &&  \;\;\;\;\;\e(\sum_{i=0}^{n-1}q^{si}1_{i} g)=1,\;\;\;\;\e(\sum_{i=0}^{n-1}q^{si}1_{i}p_{0}^{1})=0,\\
    &&  \;\;\;\;\;S(g_{2})=g_{2}^{-1},\;\;\;\;S(x)=-x\sum_{i=0}^{n-1}q^{s(i-n)}1_{i},\\
    &&  \;\;\;\;\;S(\sum_{i=0}^{n-1}q^{si}1_{i}g)=(\sum_{i=0}^{n-1}q^{si}1_{i}g)^{-1}, \\
     && \;\;\;\;\;S(\sum_{i=0}^{n-1}q^{si}1_{i}p_{0}^{1}) =-q^{(n-1)s}g_{2}^{-s}(\sum_{i=0}^{n-1}q^{si}g)^{-1}p_{0}^{1},\\
    &&  \;\;\;\;\;\alpha=g_{2}^{-s}\bi\e,\;\;\;\;\beta=1\bi \e.
      \end{eqnarray}
\end{proposition}
\begin{proof} Formulas (3.14),(3.15), (3.18) and (3.20) are clear since $D(A(n,s,q))$ contains $A(n,s,q)$ as a subquasi-Hopf algebra. By (2.29), one can verify directly that
$$\mathbf{T}(g)=g_{2}^{s}\bi g,\;\;\;\;\mathbf{T}(p_{0}^{1})=1\bi p_{0}^{1}.$$
Using the comultiplication formula ($\star\star$), we have
\begin{eqnarray}\D(\mathbf{T}(g))&=& \sum_{i,j=0}^{n-1}\q^{s[\frac{i+j}{n}]}(1_{i}\otimes 1_{j})(\mathbf{T}(g)\otimes \mathbf{T}(g)),\\
\D(\mathbf{T}(p_{0}^{1}))&=&\sum_{i,j=0}^{n-1}\q^{s[\frac{i+j}{n}]}1_{i}\mathbf{T}(p_{0}^{1})\otimes 1_{j}\\
&&+\sum_{i,j=0}^{n-1}\q^{s[\frac{i+j}{n}]-si[\frac{1+j}{n}]}1_{i}\otimes 1_{j}\mathbf{T}(p_{0}^{1})\notag.
\end{eqnarray}
From (3.24), we now know that $\sum_{i=0}^{n-1}q^{si}1_{i}\mathbf{T}(g)$ is a group-like element. Since $\mathbf{T}(g)=g_{2}^{s}\bi g$ and $g_{2}$ is group-like,
$\sum_{i=0}^{n-1}q^{si}1_{i}g$ is a group-like element. Therefore, (3.16) is proved.  By $\D$ is an algebra morphism,
$$\D(\sum_{i=0}^{n-1}q^{si}1_{i}p_{0}^{1})=\D(\sum_{i=0}^{n-1}q^{si}1_{i})\D(p_{0}^{1})=\D(\sum_{i=0}^{n-1}q^{si}1_{i})\D(\mathbf{T}(p_{0}^{1})).$$
Using (3.25) directly, one can get the formula (3.17). Once the comultiplication rule is determined, the counit is clear now. Also, by the definition of
the Drinfeld double, we know that $\alpha=g_{2}^{-s}\bi\e,\;\beta=1\bi \e$. From this and the comultiplication formulas (3.16) and (3.17), one can verify that (3.21) and (3.22) are
the desired formulas for the antipode.
\end{proof}

\textbf{Proof of Theorem 3.1. \;} Define a map
\begin{eqnarray}\Psi:\;\Q_{s}\mathbf{u}_{q}(\mathfrak{sl}_{2})\to D(A(n,s,q)),
&&g_{1}\mapsto \sum_{i=0}^{n-1}q^{si}1_{i}g,\; g_{2}\mapsto g_{2},\notag\\
&&x\mapsto x,\;\;y\mapsto \sum_{i=0}^{n-1}q^{si}1_{i}p_{0}^{1}.\notag
\end{eqnarray}
By Proposition 3.4, it is an algebra morphism. It is also surjective by Theorem 2.4 (1). Comparing the dimensions of two algebras, $\Phi$ is a bijection.
To show the result, it is enough to show that it is a coalgebra morphism. This is a direct consequence of Proposition 3.5 by noting that in the formula (3.17),
$\sum_{i=0}^{n-2}q^{si}1_{i}p_{0}^{1}=(\sum_{i=0}^{n-1}q^{si}1_{i}p_{0}^{1})\sum_{j=1}^{n-1}1_{j}$ and $q^{s(n-1)}1_{n-1}p_{0}^{1}=(\sum_{i=0}^{n-1}q^{si}1_{i}p_{0}^{1})1_{0}$.
 The theorem is proved.  $\;\;\;\;\;\;\;\;\;\;\;\;\;\;\;\;\;\;\;\;\;\square$

\section{Twist equivalence}

We give a sufficient condition to determine when $D(A(n,s,q))$ is not trivial, i.e., not twist equivalent to a Hopf algebra.

\begin{theorem} Assume that $n=2^{m}l$ and $s=2^{m'}l'$ with $(l,2)=(l',2)=1$. If $m'<m$, then $D(A(n,s,q))$ is not twist equivalent to a Hopf algebra.
\end{theorem}
\begin{proof} At first, there is no harm to assume that $n$ is divided by $s$. Secondly, we can identify $D(A(n,s,q))$ with $\Q_{s}\mathbf{u}_{q}(\mathfrak{sl}_{2})$ by Theorem 3.1.
  We construct a 1-dimensional representation for
   $\Q_{s}\mathbf{u}_{q}(\mathfrak{sl}_{2})$ through the following algebra morphism
   $$\rho:\;\Q_{s}\mathbf{u}_{q}(\mathfrak{sl}_{2})\to \kappa,\;\;g_{1}\mapsto -1,\;\;g_{2}\mapsto (-1)^{\frac{1}{s}},\;\;x\mapsto 0,\;\;y\mapsto 0.$$
   One can check directly that $\rho$ is well-defined.
   Denote this representation by $X$. Let Rep-$\Q_{s}\mathbf{u}_{q}(\mathfrak{sl}_{2})$ be the representation category of $\Q_{s}\mathbf{u}_{q}(\mathfrak{sl}_{2})$. It is
   a tensor category.

   Let $\langle X\rangle$ be the subtensor category generated by $X$.
   Explicitly, define
\begin{gather*}X^{\stackrel{\rightharpoonup}{\otimes i}}=:\stackrel{i}{\overbrace{(\cdots(X\otimes X)\otimes X)\cdots)}}.\end{gather*}
Then the objects of $\langle X\rangle$ are direct sums of elements being in $\{X^{\stackrel{\rightharpoonup}{\otimes i}}|0\leq i< 2s\}$.
   Now assume that $\Q_{s}\mathbf{u}_{q}(\mathfrak{sl}_{2})$  is twist equivalent to a Hopf algebra. By the general principle of Tannak-Krein duality (see, e.g., \cite{CE}), there is a fiber functor
   from the category Rep-$\Q_{s}\mathbf{u}_{q}(\mathfrak{sl}_{2})$ to the category of $\kappa$-spaces. Thus its restriction to $\langle X\rangle$ is still a fiber functor. This implies
   the restriction of $\phi_{s}$ to $\langle X\rangle$ should be gotten from a 3-coboundary of $\mathbb{Z}_{2s}$.
   It is not hard to see that
   \begin{gather*}\phi_{s}|_{\langle X\rangle}=\sum_{i,j,k=0}^{2s} \q^{s\frac{ni}{2s}}[\frac{\frac{nj}{2s}+\frac{nk}{2s}}{n}]1_{\frac{ni}{2s}}\otimes 1_{\frac{nj}{2s}}\otimes 1_{\frac{nk}{2s}}\\
   =   \sum_{i,j,k=0}^{2s}(-1)^{i[\frac{j+k}{2s}]}1_{\frac{ni}{2s}}\otimes 1_{\frac{nj}{2s}}\otimes 1_{\frac{nk}{2s}}.\end{gather*}
   By the general theory of
   group cohomology, it is known that the 3-cocycle
   $$f(g_{2s}^{i},g_{2s}^{j},g_{2s}^{k})=(-1)^{i[\frac{j+k}{2s}]},\;\;\;\;0\leq i,j,k< 2s$$
  is not a 3-coboundary where $g_{2s}$ denoting a generator of $\mathbb{Z}_{2s}$. That's a contradiction.
  \end{proof}

 \begin{corollary} If $n$ is even and $s$ is odd, then $D(A(n,s,q))$ is not twist equivalent to a Hopf algebra.
 \end{corollary}

 \begin{corollary} The quasitriangular quasi-Hopf algebra $D(A(n,q))$ is twist equivalent to $\mathbf{u}_{q}(\mathfrak{sl}_{2})$ if and only
 if $n$ is odd.
 \end{corollary}
 \begin{proof} As pointed out in Remark 2.3, $A(n,q)=A(n,1,q)$ and thus the ``only if" part is just a direct consequence of Corollary 4.2. The sufficiency is prove in \cite{EG} by using conceptual
 way. Here we give another proof. Let $n=2m+1$ and construct $1_{i}^{n^{2}}:=\frac{1}{n^{2}}\sum_{j=0}^{n^{2}-1}q^{-ij}(g_{1}^{m+1})^{ij}$. Define
 $$J:=\sum_{i,j=0}^{n^{2}-1}q^{i(j-j')}1_{i}^{n^{2}}\otimes 1_{j}^{n^{2}}.$$
 One can verify that $\phi_{J^{-1}}=(1\otimes J^{-1})(id\otimes \Delta)(J^{-1})\phi_{1}(\Delta\otimes id)(J)(J\otimes 1)=1\otimes 1\otimes 1$. Thus $D(A(n,q))_{J^{-1}}$ is a Hopf algebra, which is $\mathbf{u}_{q}(\mathfrak{sl}_{2})$ obviously.
 \end{proof}

  \begin{remark} \emph{(1) In particular, the quasitriangular quasi-Hopf algebra $\Q_{1}\mathbf{u}_{q}(\mathfrak{sl}_{2})$
                    is not twist equivalent to a Hopf algebra if $n$ is even. In this case, we denote it by
                    $\Q\mathbf{u}_{q}(\mathfrak{sl}_{2})$. Its representation theory will be given elsewhere.}

                    \emph{(2) The subtensor category $\langle X\rangle$ constructed in the proof of Theorem 4.1 can be realized as the representation category of the following quasi-Hopf algebra.
                        Let $\chi$ be the character of $X$ and define $I:= \bigcap_{i=0}^{2s}\Ker\chi^{i}$. $I$ is a Hopf ideal of $\Q_{s}\mathbf{u}_{q}(\mathfrak{sl}_{2})$. Then $\langle X\rangle$ is isomorphic to Rep-$\Q_{s}\mathbf{u}_{q}(\mathfrak{sl}_{2})/I$. }
    \end{remark}

\section*{Acknowledgements}
I would like thank Professor Ng Siu-Hung for many stimulating discussions, and in particular for providing the main idea and the method to
prove the Theorem 4.1.  The author is also supported by Natural Science Foundation (No. 10801069).

\end{document}